\documentclass[a4paper,12pt]{amsart}

\usepackage[left=3.0cm, right=3.0cm, top=2.50cm, bottom=2.50cm]{geometry}

\usepackage{amsmath}
\usepackage{amssymb} 
\usepackage{amsfonts}
\usepackage{enumerate}
\usepackage{graphicx}

\theoremstyle{plain}
\newtheorem{theorem}{Theorem}[section]
\newtheorem{proposition}[theorem]{Proposition}
\newtheorem{lemma}[theorem]{Lemma}
\newtheorem{corollary}[theorem]{Corollary}
\newtheorem{fact}[theorem]{Fact}

\theoremstyle{definition}
\newtheorem{definition}[theorem]{Definition}
\newtheorem{example}[theorem]{Example}
\newtheorem{remark}[theorem]{Remark}

\numberwithin{equation}{section}

\newcommand{\UP}{\blacktriangle}                              
\newcommand{\DOWN}{\blacktriangledown}                        

\begin{document}

\title[Representing pseudocomplemented Kleene algebras by tolerances]{Representing regular pseudocomplemented Kleene algebras by tolerance-based rough sets}

\author[J.~J{\"a}rvinen]{Jouni J{\"a}rvinen}
\address{J.~J{\"a}rvinen, Sirkankuja 1, 20810~Turku, Finland}

\author[S.~Radeleczki]{S{\'a}ndor Radeleczki}
\address{S.~Radeleczki, Institute of Mathematics\\ 
University of Miskolc\\3515~Miskolc-Egyetemv{\'a}ros\\Hungary}

\keywords{Kleene algebra, algebraic lattice, pseudocomplemented lattice, regular double $p$-algebra, rough set, tolerance relation, irredundant covering}

\begin{abstract} 
We show that any regular pseudocomplemented Kleene algebra defined on an algebraic lattice is isomorphic to a rough set Kleene algebra 
determined by a tolerance induced by an irredundant covering. \\

{\emph{2010 Mathematics Subject Classification.} Primary: 06B15; Secondary: 06D15, 06D30, 68T37, 06D20.}
\end{abstract}

\maketitle

\section{Introduction}

Kleene algebras were introduced by D.~Brignole and A.~Monteiro in \cite{BrigMont1967}. Earlier J.~A.~Kalman  \cite{Kalman58}
called such distributive lattices with an involution $\sim$ satisfying $x \wedge {\sim} x \leq y \vee {\sim} y$
as ``normal $i$-lattices''. Kleene algebras can be seen as generalisations of, for instance, Boolean, {\L}ukasiewicz, 
Nelson, and Post algebras; see \cite{BaDw74}. The notion used here should not to be confused with the other Kleene algebra 
notion generalizing regular expressions.

Clearly, any pseudocomplemented Kleene algebra is defined on a distributive double pseudocomplemented lattice. 
According to \cite{Sankappanavar86} a pseudocomplemented Kleene algebra $(L,\vee,\wedge,{\sim},{^*},0,1)$
is congruence-regular if and only if the distributive double $p$-algebra $(L,\vee,\wedge,{^*}, {^+}, 0,1)$ is
congruence-regular. Varlet \cite{Varlet1972} has shown that any double $p$-algebra 
is congruence-regular if and only it is determination-trivial, that is, $x^* = y^*$ and $x^+ = y^+$ imply $x = y$.
Therefore, a pseudocomplemented Kleene algebra is congruence-regular if and only if it is determination-trivial.
In this paper we study the representation of (congruence-)regular pseudocomplemented 
Kleene algebras  whose underlying lattice is algebraic.

It is well known that any Boolean algebra defined on an algebraic lattice is isomorphic to the powerset algebra $\wp(U)$ of some set $U$.  
In this paper we prove an analogous result for a regular pseudocomplemented Kleene algebra and the algebra of rough sets defined by a tolerance 
induced by an irredundant covering of a set.

Rough sets were introduced by Z.~Pawlak in \cite{Pawl82}. In rough set theory it is assumed that our knowledge
about a universe of discourse $U$ is given in terms of a binary relation reflecting the distinguishability or 
indistinguishability of the elements of $U$. According to Pawlak's original definition, the knowledge is given 
by an equivalence $E$ on $U$ interpreted so that two elements of $U$ are $E$-related if they cannot 
be distinguished by their properties known by us. Nowadays in the literature numerous studies can be found  
in which approximations are determined by other types of relations.

If $R$ is a given binary relation on $U$, then for any 
subset $X \subseteq U$,  the  \emph{lower approximation} of $X$ is defined as
\[
 X^\DOWN = \{x \in U \mid R(x) \subseteq X \}
\]
and the \emph{upper approximation of $X$} is 
\[
X^\UP = \{x \in U \mid R(x) \cap X \neq \emptyset \},
\] 
where $R(x) = \{ y \in U \mid x \, R \, y \}$. The set $X^\DOWN$ may be interpreted as the set of objects
that certainly are in $X$ in view of the knowledge $R$, because if $x \in X^\DOWN$, then
all elements to which $x$ is $R$-related are in $X$. Similarly, the set $X^\UP$ may be considered as the
set of elements that are possibly in $X$, since $x \in X^\UP$ means that there exists at least one element
in $X$ to which $x$ is $R$-related. Note that the maps $^\DOWN$ and $^\UP$ are dual, that is, 
$X^{\DOWN c} = X^{c \UP}$ and $X^{\UP c} = X^{\DOWN c}$ for all $X \subseteq U$, where
$X^c$ denotes the complement $U \setminus X$ of $X$.

The \emph{rough set} of $X$ 
is the pair $(X^\DOWN,X^\UP)$, and the set of all 
rough sets is
\[
 \mathit{RS} = \{ (X^\DOWN, X^\UP)  \mid X \subseteq U \}. 
\]
The set $\mathit{RS}$  may be canonically ordered by the
coordinatewise order: 
\[
(X^\DOWN,X^\UP) \leq (Y^\DOWN,Y^\UP) \iff X^\DOWN \subseteq Y^\DOWN \mbox{ \
and \ } X^\UP \subseteq Y^\UP. 
\]
The structure of $RS$ is well studied in the case when $R$ is an equivalence;
see \cite{Com93, duentsch97, GeWa92, Iturrioz99, Jarv07, Pagliani97, PomPom88}. 
In particular, J.~Pomyka{\l }a and J.~A.~Pomyka{\l }a showed in \cite{PomPom88} that 
$RS$ is a Stone lattice. Later this result was improved by S.~D.~Comer \cite{Com93} 
by showing that $RS$ is a regular double Stone algebra.
In  \cite{GeWa92}, M.~Gehrke and E.~Walker proved that $RS$ is isomorphic to 
$\mathbf{2}^I \times \mathbf{3}^K$, where $I$ is the set of singleton 
$R$-classes and $K$ is the set of non-singleton equivalence classes of $R$. 

If $R$ is a quasiorder (i.e., a reflexive and transitive binary relation), then $\mathit{RS}$ is
a completely distributive and algebraic lattice \cite{JRV09}. We showed in \cite{JarRad11} how 
one can define a Nelson algebra $\mathbb{RS}$ on this algebraic lattice. In addition, we proved that if
$\mathbb{L} = (L,\vee,\wedge,{\sim},\to,0,1)$ is a Nelson algebra defined on an algebraic lattice, then there 
exists a set $U$  and a quasiorder $R$ on $U$ such that $\mathbb{L}$ is isomorphic to the rough set Nelson algebra 
$\mathbb{RS}$ determined by $R$.

Let $R$ be a tolerance on $U$, that is, $R$ is a reflexive and symmetric binary relation on $U$.
The pair $({^\DOWN}, {^\UP})$ of rough approximation operations forms an order-preserving Galois connection
on the powerset lattice $(\wp(U),\subseteq)$ of $U$, that is, $X^\UP \subseteq Y \Leftrightarrow X ¸\subseteq Y^\DOWN$
for any $X,Y \subseteq U$. The essential facts about Galois connections can be found in 
\cite{ErKoMeSt93}, for instance. Because $R$ is reflexive, we have also $X^\DOWN \subseteq X \subseteq X^\UP$
for all $X \subseteq U$.  Properties of rough approximations defined by tolerances are given in \cite{Jarv99,JarRad14},
and they are not recalled here.

It is known that if $R$ is a tolerance, then $\mathit{RS}$ is not necessarily even a lattice \cite{Jarv99}. However,
we proved in \cite{JarRad14} that if $R$ is a tolerance induced by an irredundant covering of $U$, then 
$(\mathit{RS},\vee,\wedge, {\sim}, (\emptyset,\emptyset), (U,U))$ is a Kleene algebra such that 
$\mathit{RS}$ is algebraic and completely distributive. This means that $\mathit{RS}$ forms a double 
pseudocomplemented lattice. Our main result shows that if $\mathbb{L} = (L,\vee,\wedge,{\sim},{^*}, 0,1)$ is a regular pseudocomplemented Kleene
algebra defined on an algebraic lattice, then there exists a set $U$ and a tolerance $R$ induced by an irredundant
covering of $U$ such that $\mathbb{L}$ is isomorphic to the rough set pseudocomplemented Kleene algebra 
$\mathbb{RS} = (\mathit{RS},\vee,\wedge,{\sim}, {^*}, (\emptyset,\emptyset), (U,U))$  determined by $R$.

The paper is structured as follows. In the next section we recall some notions and facts related to De~Morgan, Kleene,
Heyting, and double pseudocomplemented algebras. In Section~\ref{Sec:RoughSets} we study rough sets defined by tolerances
induced by irredundant coverings. In particular, the structure of their completely join-irreducible
elements is given. Varlet \cite{Varlet1972} has proved that any distributive double pseudocomplemented lattice is regular
if and only if any chain of its prime filters has at most two elements. In Section~\ref{Sec:Regularity} we first show that 
if $\mathbb{L} = (L,\vee,\wedge,{\sim},{^*}, 0,1)$ is a pseudocomplemented De~Morgan algebra defined on an algebraic lattice, 
then $\mathbb{L}$ is regular if and only if the set of completely join-irreducible elements of $L$ has at most two levels. 
In the end of the section we consider irredundant coverings and their tolerances determined by regular pseudocomplemented Kleene 
algebras defined on algebraic lattices. Section~\ref{Sec:Representation} contains our representation theorem and its proof.
The construction is also illustrated by an example. Finally, Section~\ref{Sec:Remarks} contains some concluding remarks.

\section{Preliminaries} \label{Sec:Preliminaries}

In this section we recall some general lattice-theoretical notions and notation which can be found, for instance, in the books 
\cite{BaDw74, DaPr02, Gratzer}. For more specific results, a reference will be given.

An element $j$ of a complete lattice $L$ is called \emph{completely join-irreducible} if $j = \bigvee S$
implies $j \in S$ for every subset $S$ of $L$. Note that the least element $0$ of $L$ is not completely
join-irreducible. The set of completely join-irreducible elements of $L$ is denoted by 
$\mathcal{J}(L)$, or simply by $\mathcal{J}$ if there is no danger of confusion.
A complete lattice $L$ is \emph{spatial} if for each
$x \in L$,
\[ x = \bigvee \{ j \in \mathcal{J} \mid j \leq x \}. \]
An element $x$ of a complete lattice $L$ is said to be \emph{compact} if, for every $S \subseteq L$, 
\[ x \leq \bigvee S \Longrightarrow x \leq \bigvee F \text{ for some finite subset $F$ of $S$}. \]
Let us denote by $\mathcal{K}(L)$ the set of compact elements of $L$.
A complete lattice $L$ is said to be \emph{algebraic} if for each $a \in L$,
\[ a  = \bigvee \{ x \in \mathcal{K}(L) \mid x \leq a\}.\]  

Note that if $L$ is an algebraic lattice, then its completely join-irreducible elements are compact. 
Let the lattice $L$ be both algebraic and spatial. Since any compact element can be written as a 
finite join and any finite join  of compact elements is compact, the compact elements of $L$ are exactly 
those that can be written as a finite join of completely join-irreducible elements.

A complete lattice $L$ is \emph{completely distributive} if for any doubly indexed
subset $\{x_{i,\,j}\}_{i \in I, \, j \in J}$ of $L$, we have
\[
\bigwedge_{i \in I} \Big ( \bigvee_{j \in J} x_{i,\,j} \Big ) = 
\bigvee_{ f \colon I \to J} \Big ( \bigwedge_{i \in I} x_{i, \, f(i) } \Big ), \]
that is, any meet of joins may be converted into the join of all
possible elements obtained by taking the meet over $i \in I$ of
elements $x_{i,\,k}$\/, where $k$ depends on $i$.

A \emph{complete ring of sets} is a family $\mathcal{F}$ of sets such that $\bigcup \mathcal{S}$ and 
$\bigcap \mathcal{S}$ belong to $\mathcal{F}$ for any subfamily $\mathcal{S} \subseteq \mathcal{F}$.
\begin{remark}\label{Rem:RingOfSets}
Let $L$ be a lattice. Then the following are equivalent:
\begin{enumerate}[\rm (a)]
\item $L$ is isomorphic to a complete ring of sets;
\item $L$ is algebraic and completely distributive;
\item $L$ is distributive and doubly algebraic (i.e.\@ both $L$ and the dual $L^\partial$ of $L$ are algebraic);
\item $L$ is algebraic, distributive and spatial.
\end{enumerate}
\end{remark}

A \emph{De Morgan algebra} is an algebra $\mathbb{L} =(L,\vee,\wedge,\sim,0,1)$
of type $(2,2,1,0,0)$ such that $(L,\vee,\wedge, 0, 1)$ is a bounded distributive
lattice and $\sim$ satisfies for all $x,y\in L$,
\begin{center}
${\sim}{\sim} x =x$  \quad and \quad  $x\leq y \iff {\sim} y \leq {\sim} x$.
\end{center}
This definition means that $\sim$ is an isomorphism between the lattice $L$ and its dual $L^\partial$.
Thus for a De~Morgan algebra $\mathbb{L}$, the underlying lattice $L$ is self-dual, and for each $x,y \in L$,
\[
{\sim}(x \vee y) = {\sim} x \wedge {\sim} y \text{\quad and\quad } {\sim} (x \wedge y) = {\sim} x \vee {\sim} y.
\]

We say that a De Morgan algebra is \emph{completely distributive} if its underlying lattice is completely distributive. 
Let $\mathbb{L}$ be a completely distributive De Morgan algebra. We define for any $j \in \mathcal{J}$ the element
\begin{equation}\label{Eq:Gee}
 g(j)= \bigwedge \{x \in L \mid x\nleq {\sim} j \}.
\end{equation}
This $g(j) \in \mathcal{J}$ is the least element which is not below ${\sim}j$. The function $g \colon \mathcal{J} \to \mathcal{J}$
satisfies the conditions:
\begin{enumerate}[({J}1)]
 \item if $x \leq y$, then $g(x) \geq g(y)$;
 \item $g(g(x))= x$.
\end{enumerate}
In fact, $(\mathcal{J},\leq)$ is self-dual by the map $g$.

Let $\mathbb{L}$ be a De~Morgan algebra defined on an algebraic lattice. The underlying lattice $L$ is doubly algebraic, because it is self-dual.
Therefore, the lattice $L$ has all equivalent properties (a)--(d) of Remark~\ref{Rem:RingOfSets}. We also have that the operation ${\sim}$ is
expressed in terms of $g$ by:
\begin{equation} \label{Eq:Connection}
 {\sim} x= \bigvee \{j \in \mathcal{J} \mid g(j) \nleq x \}. 
\end{equation}
For studies on the properties of the map $g$, see  \cite{Cign86, JarRad11, Mont63a}, for example.

\medskip%

A \emph{Kleene algebra} is a De Morgan algebra $\mathbb{L}$ satisfying 
\begin{equation} \label{Eq:Kleene}\tag{K}
x \wedge {\sim} x \leq y \vee {\sim} y
\end{equation}
for each $x,y \in L$. It is proved in \cite{CigGal81} that for any Kleene algebra $\mathbb{L}$ and  $x,y \in L$, 
\begin{equation}\label{Eq:KleeneComplements}
 x \wedge y = 0 \text{\quad implies \quad } y \leq {\sim} x.
\end{equation}
If $\mathbb{L}$ is a completely distributive Kleene algebra, then $j$ and $g(j)$ are comparable for any $j \in \mathcal{J}$, that is,
\begin{enumerate}[({J}3)]
 \item [({J}3)] $g(j)\leq j \text{ or } j \leq g(j)$.
\end{enumerate}

A \emph{Heyting algebra} is a bounded lattice $L$ such that for all $a,b\in L$, there is a greatest element $x$ of $L$ satisfying
$a \wedge x \leq b$. This element $x$ is called the \emph{relative pseudocomplement}  of $a$ with respect to $b$, and it is denoted by 
$a \Rightarrow b$. Heyting algebras are not only distributive, but they satisfy the \emph{join-infinite distributive law} (JID), that is, 
if $\bigvee S$ exists for some $S \subseteq L$, then for each $x \in L$, $\bigvee \{x \wedge y \mid a \in S\}$ exists and 
$x \wedge ( \bigvee S ) = \bigvee \{ x \wedge y \mid y \in S \}$. Conversely, any complete lattice satisfying (JID) is a  Heyting algebra, with
$ a \Rightarrow b = \bigvee \{c \mid a \wedge c \le b\}$. 

A \emph{double Heyting algebra} $L$ is such that both $L$ and its dual $L^\partial$ are Heyting algebras; see \cite{Kat73}, for instance. 
This means that in $L$ there are two implications $\Rightarrow$ and $\Leftarrow$ which are fully determined by 
$\leq$, and $\Leftarrow$ satisfies $a \vee x \geq b$ if and only if $x \geq a \Leftarrow b$ for all $a,b,x \in L$.
These structures are also called \emph{Heyting--Brouwer algebras}.

A Heyting algebra $L$ such that $(L, \vee, \wedge, {\sim}, 0,1)$ is a De~Morgan algebra is called a \emph{symmetric Heyting algebra}; see \cite{Monteiro1980}.
Each symmetric Heyting algebra defines a double Heyting algebra such that $a \Leftarrow b$ equals ${\sim} ({\sim} a \Rightarrow {\sim} b)$.

\begin{example}\label{Exa:DoubleHeyting}
(a) Every De~Morgan algebra defined on an algebraic lattice determines a symmetric Heyting algebra by Remark~\ref{Rem:RingOfSets}.

(b) A complete lattice is a double Heyting algebra if and only if it satisfies (JID) and (MID), where (MID) is the dual condition of (JID). 
In particular, every finite distributive lattice is a double Heyting algebra. Of course, not every finite and distributive lattice is self-dual, 
that is, a symmetric Heyting algebra.

(c) One double Heyting algebra may define several symmetric Heyting algebras. For instance, the Boolean algebra $\mathbf{2^2}$
with $0 < a,b < 1$ is a double Heyting algebra and we can define a De~Morgan operation $\sim$ in $\mathbf{2^2}$ by two ways: either (i) $a \mapsto a$ and $b \mapsto b$;
or (ii) $a \mapsto b$ and $b \mapsto a$. These mappings are De~Morgan operations when for both cases we set ${\sim} 0 = 1$ and ${\sim} 1 = 0$.
\end{example}

In a lattice $L$ with the least element $0$, an element $x \in L$ is said to have a \emph{pseudocomplement} if there exists an
element $x^*$ in $L$ having the property that for any $z\in L$, $x\wedge z=0 \Leftrightarrow z\leq x^*$.
The lattice $L$ itself is called \emph{pseudocomplemented}, if every element of $L$ has a pseudocomplement. Every pseudocomplemented lattice is necessarily bounded, 
having $0^*$ as the greatest element. The algebra $(L, \vee, \wedge,{^*},0,1)$ is called also a \emph{$p$-algebra} for short. The following hold for 
every $a,b \in L$:
\begin{enumerate}[\rm (i)]
 \item $a \leq b$ implies $b^* \leq a^*$;
 \item the map $a \mapsto a^{**}$ is a closure operator;
 \item $a^* = a^{***}$;
 \item $(a \vee b)^* = a^* \wedge b^*$;
 \item $(a \wedge b)^* \geq a^* \vee b^*$.
\end{enumerate}

An algebra $(L,\vee,\wedge,{^*},{^+},0,1)$ is called a \emph{double $p$-algebra} if $(L,\vee,\wedge,{^*},0,1)$ is a $p$-algebra and $(L,\vee, \wedge,{^+},0,1)$ 
is a dual $p$-algebra (that is, $z \geq x^+ \Leftrightarrow x \vee z = 1$ for all $x,y \in L$). In the literature, the term \emph{double pseudocomplemented lattice} is
often used instead of double $p$-algebra. Each Heyting algebra $L$ defines a distributive $p$-algebra 
by setting $x^* := x \Rightarrow 0$, and if $L$ is also a double Heyting algebra,  it determines a distributive double $p$-algebra, where $x^+ := x \Leftarrow 1$.

For a double $p$-algebra $(L,\vee,\wedge,{^*},{^+},0,1)$, the \emph{determination congruence} $\Phi$ is defined by
\[
\Phi := \{(x, y) \mid x^* = y^* \mbox{ and  } x^+ = y^+\}.
\]
A double $p$-algebra is called \emph{determination-trivial} if $\Phi = \{ (x,x) \mid x \in L\}$. This is obviously equivalent
to the fact that the double $p$-algebra satisfies the condition:
\begin{equation}
x^* = y^* \text{ and } x^+ = y^+ \text{ imply } x=y. \tag{M}
\end{equation}
An algebra is called \emph{congruence-regular} if every congruence is determined by any class of it: 
two congruences are necessarily equal when they have a class in common. J.~Varlet has proved in  \cite{Varlet1972} that double $p$-algebras satisfying (M) 
are exactly the congruence-regular ones. In addition, T.~Katri\u{n}\'{a}k \cite{Kat73} has shown that any 
congruence-regular double pseudocomplemented lattice forms a double Heyting algebra such that
\begin{align}
a \Rightarrow b &= (a^* \vee   b^{**})^{**} \wedge   [(a \vee a^*)^+ \vee a^* \vee b \vee b^* ] ; \label{Eq:RegularHeyting1}\\
a \Leftarrow  b &= (a^+ \wedge b^{++})^{++} \vee [(a \wedge  a^+)^* \wedge a^+ \wedge b \wedge b^+ ].  \label{Eq:RegularHeyting2}
\end{align} 

A \emph{pseudocomplemented De Morgan algebra} is an algebra $(L,\vee,\wedge,{\sim}, {^*}, 0,1)$ such that
$(L,\vee,\wedge,{\sim}, 0,1)$ is a De~Morgan algebra and $(L,\vee,\wedge, {^*}, 0,1)$ is a $p$-algebra. 
In fact, such an algebra forms a double $p$-algebra, where the pseudocomplement operations
determine each other by:
\begin{equation}\label{Eq:Pseudocomplements}
{\sim} x^* = ({\sim} x)^+ \text{ \ and \ } {\sim} x^+ = ({\sim} x)^*.
\end{equation}
H.~P.~Sankappanavar has proved in \cite{Sankappanavar86} that a pseudocomplemented De Morgan algebra satisfying (M)
truly is a congruence-regular pseudocomplemented De Morgan algebra. Therefore, in the sequel we may call pseudocomplemented 
De Morgan and Kleene algebras \emph{regular} when they satisfy (M).
Note that regular pseudocomplemented De~Morgan algebras define symmetric (double) Heyting algebras, 
where the operations $\Rightarrow$ and $\Leftarrow$ are given by \eqref{Eq:RegularHeyting1} and \eqref{Eq:RegularHeyting2}. 

A pseudocomplemented De~Morgan algebra $\mathbb{L}$ is \emph{normal} (see \cite{Monteiro1980}), if for all $x \in L$, 
\begin{equation}
x^* \leq {\sim} x . \tag{N}
\end{equation}
Note that if $\mathbb{L}$ is normal, then for every $x \in L$ and $y = {\sim} x$, we have ${\sim}({\sim} y)^+ = y^* \leq {\sim} y$. 
Hence $(\sim y)^+ \geq y$ and so $x^+ \geq {\sim} x$. Thus
\[
x^* \leq {\sim} x \leq x^+.
\]
It is known (see e.g. \cite{Kat73}) that in any distributive double $p$-algebra, the ``regularity condition'' (M) is equivalent 
to the condition
\begin{equation}
x \wedge x^+ \leq y \vee y^*. \tag{D}
\end{equation}
This means that if $\mathbb{L}$ is a normal and regular pseudocomplemented De~Morgan algebra, then for any $x,y \in L$,
\[
 x \wedge {\sim} x \leq x \wedge x^+ \leq y \vee y^* \leq y \vee {\sim y}.
\]
Therefore, any normal and regular pseudocomplemented De~Morgan algebra forms a Kleene algebra. On the other hand, any pseudocomplemented 
Kleene algebra is normal by \eqref{Eq:KleeneComplements}. Hence we obtain:

\begin{remark} \label{Rem:RegularNormal}
Any regular pseudocomplemented De~Morgan algebra is a Kleene algebra if and only if it is normal.
\end{remark}

A filter $F$ of a lattice $L$ is called \emph{proper}, if $F \ne L$. A proper filter $F$ is a \emph{prime filter} if $a \vee b \in F$ implies 
$a \in F$ or $b \in F$. The set of prime filters of $L$ is denoted by $\mathcal{F}_p(L)$, or by $\mathcal{F}_p$ if there is no danger of confusion. 
Proper ideals and prime ideals are defined analogously. Clearly, $F$ is a prime filter if and only if $L \setminus F$ is a prime ideal.

If $L$ is a bounded distributive lattice, then any prime filter $F$ is contained in a maximal prime filter. 
Moreover, any maximal prime filter is a maximal proper filter.
If $L$ is a distributive lattice, then the principal filter $[j) = \{x \in L \mid x \geq j\}$ of each $j \in \mathcal{J}$ is prime. 
For any $j \in \mathcal{J}$, the prime filter $[j)$ is maximal if and only if $j$ is an atom. 

\begin{fact}\label{Fact:TwoLevel}
For any bounded distributive lattice, the following are equivalent:
\begin{enumerate}[\rm (a)]
 \item There is no 3-element chain of prime filters;
 \item For any $P, Q \in \mathcal{F}_p$ , $P \subset Q$ implies that $Q$ is a maximal filter.
\end{enumerate}
\end{fact}

\begin{proposition}[\cite{Varlet1972}]\label{Prop:Varlet}
Let $(L, \vee, \wedge, {^*}, {^+}, 0, 1)$ be a distributive double $p$-algebra. The following are equivalent:
\begin{enumerate}[\rm (i)]
 \item $L$ is regular;
 \item Any chain of prime filters (or ideals) of $L$ has at most two elements.
\end{enumerate}
\end{proposition}

\section{Rough sets defined by tolerances induced by irredundant coverings}  \label{Sec:RoughSets}

Let $R$ be a tolerance on $U$. A nonempty subset $X$ of $U$ is a \emph{preblock} if $X^2 \subseteq R$. 
A \emph{block} is a maximal preblock, that is, a preblock $B$ is a block if $B \subseteq X$ implies $B = X$ for any preblock $X$. 
Thus any subset $\emptyset \neq X \subseteq U$ is a preblock if and only if it is contained in some block.
Each tolerance $R$ is completely determined by its blocks, that is, $a \, R \, b$ if and only if there exists a block 
$B$ such that $a,b \in B$. 

A collection $\mathcal{H}$ of nonempty subsets of $U$ is called a \emph{covering} of $U$ if $\bigcup \mathcal{H} = U$.
A covering $\mathcal{H}$ is \emph{irredundant} if $\mathcal{H} \setminus \{X\}$ is not a covering for any $X \in \mathcal{H}$.
Each covering $\mathcal{H}$ defines a tolerance $R_\mathcal{H} = \bigcup \{ X^2 \mid X \in \mathcal{H}\}$, 
called the \emph{tolerance induced} by $\mathcal{H}$. Obviously, the sets in $\mathcal{H}$ are preblocks of $R_\mathcal{H}$ and
$R_\mathcal{H}(x) = \bigcup \{ B \in \mathcal{H} \mid x \in B \}$.  
Thus $x \in B$ implies $B \subseteq R_\mathcal{H}(x)$ for any $B \in \mathcal{H}$. 

We proved in \cite{JarRad14} that $\mathcal{H}$ is irredundant if and only if $\mathcal{H} \subseteq \{R_\mathcal{H}(x) \mid x \in U\}$. In addition,
if $\mathcal{H}$ is irredundant, then $\mathcal{H}$ consists of blocks of $R_\mathcal{H}$ \cite{JarRad15}. 
We can now write the following lemma which states that each irredundant covering $\mathcal{H}$ consists of such 
$R_\mathcal{H}(x)$-sets that are blocks of $R_\mathcal{H}$. Therefore, we may simply speak about tolerances induced by an irredundant covering 
without specifying the covering in question.

\begin{lemma}\label{Lem:irreducibleCovering}
Let $R$ be a tolerance induced by an irredundant covering $\mathcal{H}$ of $U$. 
Then $\mathcal{H} = \{ R(x) \mid \text{$R(x)$ is a block} \}$.
\end{lemma}

\begin{proof}
By the above, $\mathcal{H} \subseteq  \{ R(x) \mid \text{$R(x)$ is a block} \}$. On the other hand, suppose that $R(x)$ is a block.
Because $\mathcal{H}$ is a covering, there is $B \in \mathcal{H}$ such that $x \in B$. This gives that $B \subseteq R(x)$. Since
$\mathcal{H}$ is irredundant, $B$ is a block.
Because both $R(x)$ and $B$ are blocks, $B = R(x)$ and $R(x) \in \mathcal{H}$.
\end{proof}

Let $L$ be a lattice with the least element $0$. An element $a$ is an \emph{atom} of $L$ if it covers $0$, that is, $0 \prec a$.  
We denote by $\mathcal{A}(L)$ the set of atoms of $L$, and simply by $\mathcal{A}$ if there is no danger of confusion.
The lattice $L$ is \emph{atomistic}, if $x = \bigvee \{a \in \mathcal{A} \mid a \leq x\}$ for all $x \in L$. It is well known that 
a Boolean lattice is atomistic if and only if it is completely distributive; see \cite{Gratzer}, for example.

Let $R$ be tolerance on $U$. In \cite{Jarv99} it is proved that $({^\UP}, {^\DOWN})$ is an order-preserving
Galois connection on the complete lattice $(\wp(U),\subseteq)$. This implies that 
$\wp(U)^\DOWN = \{ X^\DOWN \mid X \subseteq U \}$ is a complete lattice such that 
\[ \bigwedge \mathcal{H} = \bigcap \mathcal{H} \text{ \ and \ } \bigvee \mathcal{H} = \big ( \bigcup \mathcal{H} \big )^{\UP \DOWN} \]
for all $\mathcal{H} \subseteq \wp(U)^\DOWN$. Similarly, $\wp(U)^\UP = \{ X^\UP \mid X \subseteq U\}$ is is a complete lattice
in which for all $\mathcal{H} \subseteq \wp(U)^\UP$,
\[ \bigwedge \mathcal{H} = \big ( \bigcap \mathcal{H} \big )^{\DOWN \UP} \text{ \ and \ } \bigvee \mathcal{H} = \bigcup \mathcal{H}. \]
Because  $({^\UP}, {^\DOWN})$ is a Galois connection, $(\wp(U)^\DOWN,\subseteq)$ and $(\wp(U)^\UP,\subseteq)$ are isomorphic.
In \cite{JarRad14} we proved that if $R$ is a tolerance induced by an irredundant covering, $\wp(U)^\DOWN$ and  $\wp(U)^\UP$ are atomistic Boolean lattices such that 
$\{ R(x)^\DOWN \mid R(x) \text{ is a block}\, \}$ and $\{ R(x) \mid R(x) \text{ is a block}\, \}$ are their sets of atoms, respectively. 
By Lemma~\ref{Lem:irreducibleCovering},   $\{ R(x) \mid R(x) \text{ is a block} \, \}$ is the unique irredundant covering
inducing $R$. The Boolean complement operation in $\wp(U)^\DOWN$ is $X^\DOWN \mapsto X^{\DOWN c \DOWN}$ and the complement operation 
in $\wp(U)^\UP$ is $X^\UP \mapsto X^{\UP c \UP}$.

\begin{samepage}
\begin{lemma} \label{Lem:AppByCovering}
Let $R$ be a tolerance induced by a covering $\mathcal{H}$ of $U$, $B \in \mathcal{H}$, and $X \subseteq U$. Then:
\begin{enumerate}[\rm (a)]
\item $X^\UP = \bigcup \{B \in \mathcal{H} \mid X \cap B \neq \emptyset\}$.
\item $B^\DOWN = \{x \in U \mid R(x) = B\}$.
\item If $\mathcal{H}$ is irredundant, then $\emptyset \neq B^\DOWN = B \setminus \bigcup (\mathcal{H} \setminus \{B\})$.
\end{enumerate}
\end{lemma}
\end{samepage} \enlargethispage{4mm}

\begin{proof} (a) The proof can be found in \cite{Jarv99}.

(b) Let $B \in \mathcal{H}$. If $x \in B^\DOWN$, then $R(x) \subseteq B$. Since $B$ is a block, $x \in B$ implies $B \subseteq R(x)$. Thus $R(x) = B$.
On the other hand, $R(x) = B$ gives $x \in B^\DOWN$.
 
(c) Suppose that $\mathcal{H}$ is irredundant and $B \in \mathcal{H}$. Then $X := B \setminus \bigcup (\mathcal{H} \setminus \{B\})$ is nonempty. 
We prove that $X = B^\DOWN$. Let $x \in X$ and $y \in R(x)$. 
Since $x \, R \, y$, there exists $C \in \mathcal{H}$ such that $x,y \in C$. 
If $C \neq B$, then $x \in  \bigcup (\mathcal{H} \setminus \{B\})$ and $x \notin X$, a contradiction. 
Therefore, $C = B$ and $y \in B$. Thus $R(x) \subseteq B$ and $x \in B^\DOWN$.
Conversely, let $x \in B^\DOWN$. Suppose that $x \in \bigcup (\mathcal{H} \setminus \{B\})$.
Then there exists $C \neq B$ in $\mathcal{H}$ such that $x \in C$. But $x \in C$ implies $C \subseteq R(x) \subseteq B$ which is not possible, because 
$\mathcal{H}$ is irredundant. Therefore, $x \in X$ and the proof is complete.
\end{proof}

We studied in \cite{JarRad14} the lattice-theoretical properties of $$\mathit{RS} = \{ (X^\DOWN, X^\UP) \mid X \subseteq U \}.$$ 
Let us recall here some of these results. We showed that $\mathit{RS}$ is a complete lattice if and only if $\mathit{RS}$ is a complete sublattice of 
the product $\wp(U)^\DOWN \times \wp(U)^\UP$.  This means that if $\mathit{RS}$ is a complete lattice, then for $\{(A_i,B_i)\}_{i \in I} \subseteq \mathit{RS}$,
\begin{equation}\label{Eq:RS_lattice_operations}
\bigwedge_{i \in I} (A_i,B_i) = \big ( \bigcap_{i \in I} A_i, \big ( \bigcap_{i \in I} B_i \big )^{\DOWN \UP} \big ) 
\text{ \ and \ }
\bigvee_{i \in I} (A_i,B_i) = \big ( \big ( \bigcup_{i \in I} A_i \big )^{\UP \DOWN}, \bigcup_{i \in I} B_i \big  ) .
\end{equation}
In addition, we proved that $\mathit{RS}$ is an algebraic and completely distributive lattice if and only if 
$R$ is induced by an irredundant covering. We also noted that if $R$ is a tolerance induced by an irredundant covering of $U$, 
then the algebra 
\[ (\textit{RS}, \vee, \wedge, {\sim}, (\emptyset,\emptyset), (U,U) ) \]
is a Kleene algebra such that the operations $\vee$ and $\wedge$ are defined as in \eqref{Eq:RS_lattice_operations} and
\[ {\sim} (X^\DOWN,X^\UP) = (X^{c \DOWN}, X^{c \UP}) = (X^{\UP c}, X^{\DOWN c}) .\]
Because $\mathit{RS}$ is a self-dual algebraic lattice, it is spatial by Remark~\ref{Rem:RingOfSets}. 
In addition, $\mathit{RS}$ forms a double $p$-algebra and a symmetric (double) Heyting algebra. 
Our next lemma describes the pseudocomplements and the dual pseudocomplements in $\mathit{RS}$.

\begin{lemma}
Let $R$ be a tolerance induced by an irredundant covering. For any $(A,B) \in \mathit{RS}$,
\[
(A,B)^* = (B^{c \DOWN}, B^{c \UP})  \text{\qquad and \qquad} (A,B)^+ = (A^{c \DOWN}, A^{c \UP}) .
\]
\end{lemma}

\begin{proof}
The lattice operations of $\mathit{RS}$ are described in \eqref{Eq:RS_lattice_operations}. Recall also that $({^\UP}, {^\DOWN})$ is a Galois connection.
First, we show that $(A,B) \wedge  (B^{c \DOWN}, B^{c \UP}) = (A \cap B^{c \DOWN}, (B \cap B^{c \UP})^{\DOWN \UP})$ equals
$(\emptyset,\emptyset)$. It suffices to show that the right component $(B \cap B^{c \UP})^{\DOWN \UP}$ is $\emptyset$, because 
then necessarily the left component $A \cap B^{c \DOWN}$ is empty. Indeed,
$(B \cap B^{c \UP})^{\DOWN \UP} =  (B^\DOWN \cap B^{c \UP \DOWN})^{\UP} =  (B^\DOWN \cap B^{\DOWN \UP c})^{\UP}
\subseteq  (B^{\DOWN \UP} \cap B^{\DOWN \UP c})^{\UP} = \emptyset^\UP = \emptyset$.

On the other hand, if $(A,B) \wedge (X,Y) = \emptyset$ for some $(A,B) \in \mathit{RS}$, then $B \wedge Y = \emptyset$ 
in the corresponding Boolean lattice $\wp(U)^\UP$. This gives $Y \subseteq B^{c \UP}$. Because
$X = Z^\DOWN$ and $Y = Z^\UP$ for some $Z \subseteq U$, we have $X^\UP = Z^{\DOWN \UP}
\subseteq Z \subseteq Z^{\UP \DOWN} = Y^\DOWN$. This implies $X^{\UP \UP} \subseteq Y^{\DOWN \UP}
\subseteq Y \subseteq B^{c \UP}$. We obtain $X^\UP \subseteq (X^\UP)^{\UP \DOWN} \subseteq B^{c \UP \DOWN}$.
Now $B \in \wp(U)^\UP$ implies $B^c \in \wp(U)^\DOWN$. Hence $B^{c \UP \DOWN} = B^c$ and we get
$X \subseteq X^{\UP \DOWN} \subseteq B^{c \DOWN}$. Therefore, $(A,B) \leq  (B^{c \DOWN}, B^{c \UP})$.

The other claim for  $(A,B)^+$ can be proved similarly. 
\end{proof}
\noindent%
Now the rough set algebra
\[ \mathbb{RS} = (\textit{RS}, \vee, \wedge, {\sim}, ^{*}, (\emptyset,\emptyset), (U,U) ) \]
is a pseudocomplemented Kleene algebra.

\begin{proposition}\label{Prop:Regularity}
If $R$ is a tolerance induced by an irredundant covering, then the pseudocomplemented Kleene algebra $\mathbb{RS}$ is regular.
\end{proposition}

\begin{proof}
We show that condition (M) holds. If $(A,B)^* = (C,D)^*$, then $B^{\DOWN c} =  B^{c \UP} = D^{c \UP} = D^{\DOWN c}$. 
So $B^\DOWN  = D^\DOWN$ and  $B^{\DOWN \UP}  = D^{\DOWN\UP}$. Because $B,D \in \wp(U)^\UP$,
$B = B^{\DOWN \UP}  = D^{\DOWN\UP} = D$. Similarly, $(A,B)^+ = (C,D)^+$ implies $A = C$.
We have proved that $(A,B) = (C,D)$.
\end{proof}

Let $R$ be a tolerance on $U$ induced by an irredundant covering. By Remark~\ref{Rem:RegularNormal}, the pseudocomplemented
Kleene algebra $\mathbb{RS}$ is normal. This means that for all $(A,B) \in \mathit{RS}$,
\[
  (A,B)^* \leq {\sim} (A,B) \leq (A,B)^+ .
\]
The elements $(X^\DOWN, X^\UP) \Rightarrow (Y^\DOWN, Y^\UP)$ and $(X^\DOWN,X^\UP) \Leftarrow (Y^\DOWN,Y^\UP)$ can be computed as in \eqref{Eq:RegularHeyting1} and 
\eqref{Eq:RegularHeyting2}. 
It is well known that for any distributive $p$-algebra $L$, the \emph{skeleton} $S^*(L) = \{a^* \mid a \in L\}$ forms a Boolean
algebra $(S^*(L),\sqcup,\wedge,{^*},0,1)$, where $a \sqcup b = (a^* \wedge b^*)^*$.
If $L$ is a distributive double $p$-algebra, also the \emph{dual skeleton} $S^+(L) = \{a^+ \mid a \in L\}$ forms
a Boolean algebra $(S^+(L),\vee,\sqcap,{^+},0,1)$, where $a \sqcap b = (a^+ \vee b^+)^+$.
We may now define $S^*(\mathit{RS}) = \{ (B^{c \DOWN}, B^{c \UP}) \mid B \in \wp(U)^\UP\}$ and 
$S^+(\mathit{RS}) = \{ (A^{c \DOWN}, A^{c \UP}) \mid A \in \wp(U)^\DOWN\}$. 

\begin{lemma} \label{Lem:Skeleton}
If $R$ is a tolerance induced by an irredundant covering, then
\begin{center}
$(\wp(U)^\UP,\supseteq) \cong (S^*(\mathit{RS}),\leq)$ \ and \ $(\wp(U)^\DOWN,\supseteq) \cong  (S^+(\mathit{RS}),\leq)$.
\end{center}
\end{lemma}

\begin{proof} We prove that the map $\varphi \colon B \mapsto (B^{c \DOWN}, B^{c \UP})$ is an order-isomorphism from $(\wp(U)^\UP,\supseteq)$ to 
$(S^*(\mathit{RS}),\leq)$. If $B, C \in \wp(U)^\UP$ and $B \supseteq C$, then $B^c \subseteq C^c$. This gives $B^{c \DOWN} \subseteq C^{c \DOWN}$ and  
$B^{c \UP} \subseteq C^{c \UP}$, that is, $(B^{c \DOWN}, B^{c \UP}) \leq (C^{c \DOWN}, C^{c \UP})$ in $\mathit{RS}$. Conversely,
$(B^{c \DOWN}, B^{c \UP}) \leq (C^{c \DOWN}, C^{c \UP})$ implies $B^{\DOWN c} = B^{c \UP} \subseteq C^{c \DOWN} = C^{\UP c}$ and
$B^\DOWN \supseteq C^\DOWN$. Therefore, also $B^{\DOWN \UP} \supseteq C^{\DOWN \UP}$. Because  $B, C \in \wp(U)^\UP$, 
$B = B^{\DOWN \UP}$ and $C = C^{\DOWN \UP}$. Thus $B \supseteq C$. The map $\varphi$ is trivially onto.

Similarly, we can show that map $A \mapsto (A^{c \UP}, A^{c \UP})$ is an order-isomorphism from $(\wp(U)^\DOWN,\supseteq)$ to 
$(S^+(\mathit{RS}),\leq)$. 
\end{proof}

Note that all lattices mentioned in Lemma~\ref{Lem:Skeleton} are as Boolean lattices also dually isomorphic with itself. 
Our next proposition describes the set of completely join-irreducible elements of $\mathit{RS}$.

\begin{proposition}\label{Prop:JoinIrreducibles}
Let $R$ be a tolerance induced by an irredundant covering. Then
\begin{align*}
\mathcal{J}(\mathit{RS}) & = \{ (R(x)^{\DOWN},R(x)^{\UP}) \mid \text{$R(x)$ is a block} \} \\
            & \cup  \{ (\emptyset,R(x)) \mid \text{$R(x)$ is a block and $|R(x)| \geq 2$} \}. 
\end{align*}
\end{proposition}

\begin{proof} Let $x$ be an element such that $R(x)$ is an block. If $|R(x)| \geq 2$, then $(\{x\}^\DOWN, \{x\}^\UP) = 
(\emptyset, R(x)) \in \mathit{RS}$. Since $R(x)$ is an atom of $\wp(U)^\UP$,
$(\emptyset, R(x))$ is an atom of $\mathit{RS}$. Trivially, all atoms are completely join-irreducible.

Next, we show that if $R(x)$ is a block, then $(R(x)^{\DOWN},R(x)^{\UP})$ is completely join-irreducible. Assume
$(R(x)^\DOWN, R(x)^\UP) = \bigvee \{ ({X_i}^\DOWN, {X_i}^\UP) \mid i \in I\}$. This means that
${X_i}^\DOWN \subseteq R(x)^\DOWN$ for all $i \in I$. Because each ${X_i}^\DOWN$ belongs to $\wp(U)^\DOWN$ and
$R(x)^\DOWN$ is an atom of $\wp(U)^\DOWN$, $\emptyset \subseteq {X_i}^\DOWN \subseteq R(x)^\DOWN$ implies that
every ${X_i}^\DOWN$ equals either to $\emptyset$ or to $R(x)^\DOWN$. But since $x \in R(x)^\DOWN = \{x\}^{\UP \DOWN}$,
each ${X_i}^\DOWN$ cannot be $\emptyset$. Therefore, there exists $k \in I$ such that  ${X_k}^\DOWN = R(x)^\DOWN$.
We have $R(x) = \{x\}^{\UP} = \{x\}^{\UP \DOWN \UP} = R(x)^{\DOWN \UP} = {X_k}^{\DOWN \UP} \subseteq X_k$.
Thus $R(x)^\UP \subseteq {X_k}^\UP$. By assumption, ${X_i}^\UP \subseteq R(x)^\UP$ for all $i \in I$.
Hence $R(x)^\UP = {X_k}^\UP$ and $(R(x)^\DOWN, R(x)^\UP) = ({X_k}^\DOWN, {X_k}^\UP)$.
 
On the other hand, suppose that $(X^\DOWN, X^\UP)$ is a completely join-irreducible element of $\mathit{RS}$. In \cite[Remark 4.6]{JarRad14}
we proved that each element of $\mathit{RS}$ can be represented as the join of a subset of 
\[
 \{ (R(x)^{\DOWN},R(x)^{\UP}) \mid x \in U \} \cup \{ (\emptyset,R(x)) \mid |R(x)| \geq 2 \} .
\]
But since $(X^\DOWN, X^\UP)$ is itself a completely join-irreducible element, we have that 
$(X^\DOWN, X^\UP) =  (R(x)^{\DOWN},R(x)^{\UP})$ for some $x \in U$, or
$(X^\DOWN, X^\UP) = (\emptyset,R(x))$ for some $x$ such that $|R(x)| \geq 2$. 

Let us assume first that $(X^\DOWN, X^\UP) =  (R(x)^{\DOWN},R(x)^{\UP})$ for some $x \in U$.
Because $R(x) \in \wp(U)^\UP$, there is a set $\{x_i\}_{i \in I} \subseteq U$ such that
$R(x) = \bigcup_{i \in I} R(x_i)$ and each $R(x_i)$ is a block. This gives
$R(x)^\UP = \big ( \bigcup_{i \in I} R(x_i) \big )^\UP =  \bigcup_{i \in I} R(x_i)^\UP$.
Analogously, $(\bigcup_{i \in I} R(x_i)^\DOWN)^{\UP \DOWN} = (\bigcup_{i \in I}  R(x_i)^{\DOWN \UP})^\DOWN
= (\bigcup_{i \in I} R(x_i))^\DOWN = R(x)^\DOWN$. This means that
\[
 (R(x)^\DOWN, R(x)^\UP) = \bigvee_\mathit{RS} \{ (R(x_i)^\DOWN, R(x_i)^\UP) \mid i \in I\}.
\]
Since $(R(x)^\DOWN, R(x)^\UP)$ is completely join-irreducible, we have that 
$(R(x)^\DOWN, R(x)^\UP) = (R(x_k)^\DOWN, R(x_k)^\UP)$ for some block $R(x_k)$.

Secondly, if $(X^\DOWN, X^\UP) = (\emptyset,R(x))$ for some $x$ such that $|R(x)| \geq 2$, then
$R(x) = \bigcup_{i \in I} R(x_i)$ for some index set $I$ such that each $R(x_i)$ is a block.
Because $R(x_i) \subseteq R(x)$ for all $i \in I$, we have $x \, R \, x_i$ for all $i \in I$.
If $x \neq x_i$, then $x,x_i \in R(x_i)$ means $|R(x_i)| \geq 2$, 
and if $x = x_i$, the assumption $|R(x)| \geq 2$ gives $|R(x_i)| \geq 2$. Therefore, each $(\emptyset,R(x_i) )$ is in $\mathit{RS}$ and
\[
(\emptyset,R(x)) = \bigvee_\mathit{RS} \{ (\emptyset, R(x_i) ) \mid i \in I\}.
\]
But since $(\emptyset, R(x))$ is completely join-irreducible by assumption, we have that 
$(\emptyset, R(x)) = (\emptyset,R(x_k))$ for some $k \in I$ such that $|R(x_k)| \geq 2$ and
$R(x_k)$ is a block.
\end{proof}

For a tolerance induced by an irredundant covering $\mathcal{H}$, we can express the completely join-irreducible elements of  
$\mathit{RS}$ also using elements of $\mathcal{H}$:
\[
\mathcal{J}(\mathit{RS}) = \{ (B^\DOWN,B^\UP) \mid B \in \mathcal{H}\} \cup \{ (\emptyset,B) \mid B \in \mathcal{H} \text{ and } |B|\geq 2 \}.
\]
Recall that $B^\DOWN$ and $B^\UP$ are given in terms of the irredundant covering $\mathcal{H}$ in Lemma~\ref{Lem:AppByCovering}.
Note also that since $\mathit{RS}$ is spatial, its every element can be described as the join of some elements in $\mathcal{J}(\mathit{RS})$.

Because $\mathbb{RS}$ is a completely distributive Kleene algebra and for any $(A,B)$, ${\sim} (A,B) = (B^c, A^c)$, we can by \eqref{Eq:Gee} define
the map $g \colon \mathcal{J}(\textit{RS}) \to \mathcal{J}(\textit{RS})$ by setting
\[ g((C,D)) = \bigwedge \{ (X^\DOWN, X^\UP) \mid (X^\DOWN, X^\UP) \nleq (D^c, C^c) \} \]
for any $(C,D) \in \mathcal{J}(RS)$.

\begin{lemma} \label{Lem:Brothers}
Let $R$ be a tolerance induced by an irredundant covering. 
\begin{enumerate}[\rm (a)]
\item If $R(x)$ is a block such that $|R(x)| \geq 2$, then
 \begin{center}
  $g((\emptyset, R(x))) = (R(x)^\DOWN, R(x)^\UP)$ \ and \ $g((R(x)^\DOWN,R(x)^\UP)) = (\emptyset, R(x))$.
 \end{center}
\item If $R(x) = \{x\}$, then $g((\{x\},\{x\})) =  (\{x\},\{x\})$.
\end{enumerate} 
\end{lemma}

\begin{proof} (a) Suppose that $R(x)$ is a block such that $|R(x)| \geq 2$. Now 
$(X^\DOWN, X^\UP) \nleq (R(x)^c, U)$ is equivalent to $X^\DOWN \nsubseteq R(x)^c$. This
means that $X^\DOWN \cap R(x) \neq \emptyset$. Thus there is $y \in X^\DOWN \cap R(x)$.
Because $y \in R(x)$, we have $R(x) \subseteq R(y)$, and $y \in X^\DOWN$ yields 
$R(x) \subseteq R(y) \subseteq X$. Thus $R(x)^\DOWN \subseteq X^\DOWN$ and 
$R(x)^\UP \subseteq X^\UP$. Therefore, $(R(x)^\DOWN, R(x)^\UP) \leq 
\bigwedge \{ (X^\DOWN, X^\UP) \mid (X^\DOWN, X^\UP) \nleq (R(x)^c, U) \} = g((\emptyset, R(x)))$.

On the other hand, $x \in \{x\}^{\UP\DOWN} =  R(x)^\DOWN \nsubseteq R(x)^c$ gives $(R(x)^\DOWN,R(x)^\UP) \nleq (R(x)^c,U)$
and $g((\emptyset,R(x))) \leq (R(x)^\DOWN,R(x)^\UP)$. Thus $g((\emptyset, R(x))) = (R(x)^\DOWN,R(x)^\UP)$.
Because $g(g(j)) = j$ for any $j \in \mathcal{J}(RS)$, we have that 
$g((R(x)^\DOWN,R(x)^\UP)) = (\emptyset, R(x))$.

(b) If $R(x) = \{x\}$, then $R(x)$ is a block and $R(x)^\DOWN = R(x)^\UP = R(x)$, because $x$ is $R$-related only to itself. 
Therefore, $(\{x\},\{x\}) \in \mathcal{J}$. Now ${\sim}(\{x\},\{x\}) = (\{x\}^c, \{x\}^c)$
and $(X^\DOWN, X^\UP) \nleq (\{x\}^c, \{x\}^c)$ holds if and only if $x \in X^\DOWN ¸\subseteq X^\UP$. This implies that
$(\{x\}, \{x\}) \leq  \bigwedge \{ (X^\DOWN, X^\UP) \mid (X^\DOWN, X^\UP) \nleq (\{x\}^c, \{x\}^c) \} = g((\{x\},\{x\}))$.
On the other hand, $(\{x\}, \{x\}) \nleq  (\{x\}^c, \{x\}^c)$ implies $g((\{x\},\{x\})) \leq (\{x\},\{x\})$.
Thus we have $g((\{x\},\{x\})) = (\{x\},\{x\})$.
\end{proof}

In the next section (see Lemma~\ref{Lem:PartitionOfAtoms}), we will show that if $\mathbb{L}$ is a regular pseudocomplemented Kleene algebra
defined on an algebraic lattice, then $x \in \mathcal{J}$ is an atom if and only if $x \leq g(x)$. Therefore, by Lemma~\ref{Lem:Brothers},
\begin{align} 
\mathcal{A}(\mathit{RS}) & = \{ ( \{x\},\{x\}) \mid R(x) = \{x\} \} \label{Eq:Atoms}\\
			 & \cup  \{ (\emptyset,R(x)) \mid \text{$R(x)$ is a block and $|R(x)| \geq 2$} \}. \nonumber
\end{align}
Note that \eqref{Eq:Atoms} can be seen also directly. Let $R(x)$ be a block. If $|R(x)| \geq 2$, then we have already
seen in the proof of Proposition~\ref{Prop:JoinIrreducibles} that $(\emptyset,R(x))$ is an atom. Obviously, the
completely join-irreducible element $(R(x)^\DOWN, R(x)^\UP)$ cannot now be an atom. If $R(x) = \{x\}$,
then  $(R(x)^\DOWN, R(x)^\UP) = (\{x\},\{x\})$ is an atom, because there is no element $(\emptyset, \{x\})$ in $\mathit{RS}$.

Each equivalence relation $E$ on $U$ is ``induced'' by the irredundant covering $U/E$ which consists of the equivalence classes of $E$.
The covering $U/E$ forms a partition of $U$, that is, the sets in $U/E$ do not intersect. The following lemma presents equivalent conditions 
for such ``isolated blocks'' in case of tolerances induced by irredundant coverings.

\begin{lemma} Let $R$ be a tolerance induced by an irredundant covering. For each $R(x)$ that is a block, the following are equivalent:
\begin{enumerate}[\rm (a)]
 \item $R(y) = R(x)$ for all $y \in R(x)$;
 \item $(R(x)^\DOWN, R(x)^\UP) = (R(x), R(x))$;
 \item Either $R(x) = \{x\}$ or $(\emptyset, R(x))$ is the only atom of $\mathit{RS}$ below $(R(x)^\DOWN, R(x)^\UP)$.
\end{enumerate}

\begin{proof} (a) $\Rightarrow$ (b): If $R(y) = R(x)$ for all $y \in R(x)$, then $y \in R(x)$ yields $y \in R(x)^\DOWN$. 
Thus $R(x)^\DOWN = R(x)$. This implies $R(x) \subseteq R(x)^\UP = R(x)^{\DOWN \UP} \subseteq R(x)$ and $R(x)^\UP = R(x)$.

(b) $\Rightarrow$ (c): If $|R(x)| = 1$, then $R(x) = \{x\}$. Suppose that $|R(x)| \geq 2$. By \eqref{Eq:Atoms}, $(\emptyset, R(x))$ is an atom of $\mathit{RS}$.
Clearly, $(\emptyset, R(x)) < (R(x)^\DOWN,R(x)^\UP)$. Assume $(X^\DOWN,X^\UP)$ is an atom below  $(R(x)^\DOWN, R(x)^\UP)$.
Then, by \eqref{Eq:Atoms}, we have that either (i) $X^\DOWN = X^\UP = \{y\}$ for some $y$ such that $R(y) = \{y\}$ or
(ii) $X^\DOWN = \emptyset$ and $X^\UP = R(y)$ for some $y$ such that $R(y)$ is a block having at least two elements.
(i)~If $R(y) = \{y\}$, then $\{y\}\subseteq R(x)$ gives that $y \, R \, x$, and hence
$y=x$. However, this is impossible, because $|R(x)| \geq 2$ and $|R(y)| = 1$. (ii)~If $X^\DOWN = \emptyset$ and 
$X^\UP = R(y)$ for some $y$  such that $R(y)$ is a block having at least two elements, then $(\emptyset,R(y)) \leq (R(x), R(x))$ 
gives $R(y) \subseteq R(x)$. Because $R(x)$ and $R(y)$ are blocks, we have $R(x) = R(y)$.

(c) $\Rightarrow$ (a): If $R(x) = \{x\}$, then obviously (a) is satisfied. On the other hand,
suppose that $(\emptyset, R(x))$ is the only atom of $\mathit{RS}$ below $(R(x)^\DOWN, R(x)^\UP)$. 
If $y \in R(x)$, then $R(x) \subseteq R(y)$, because $R(x)$ is a block. We are going to
prove that $R(x) = R(y)$. Assume by contraction that $R(x)\subset R(y)$. This means that there
is $z \in R(y) \setminus R(x)$. Since $z \, R \, y$, there is $w \in U$ such that $R(w)$ is a block, $R(w) \neq R(x)$, and $z,y \in R(w)$.
Therefore, $|R(w)| \geq 2$. Because $y \in R(w)$, $R(w) \subseteq R(y)$. On the other hand, $x \in R(x) \subset R(y)$ gives $y \in R(x)$
and $R(y) = \{y\}^\UP \subseteq R(x)^\UP$. Therefore, $(\emptyset, R(w)) \leq (R(x)^\DOWN, R(x)^\UP)$. We have that $R(x) = R(w)$
and $z \in R(x)$, a contradiction.
\end{proof}
\end{lemma}

\section{Regularity in pseudocomplemented Kleene algebras} 
\label{Sec:Regularity}

In this section we study the structure of completely join-irreducible elements of Kleene algebras defined on algebraic lattices.
Such algebras define pseudocomplemented Kleene algebras and we will prove that such an algebra is regular if and only
if the set $\mathcal{J}$ of the completely join-irreducible elements have at most two levels. The obtained results are
used in defining irredundant coverings and their tolerances.

A lattice $L$ with $0$ is called \emph{atomic}, if for any $x \neq 0$ there exists an atom $a \leq x$. 
Clearly, every atomistic lattice is atomic.

\begin{definition}
The set of completely join-irreducible elements of complete lattice \emph{has at most two levels}, if for any completely
join-irreducible elements $j$ and $k$, $j < k$ implies that $j$ is an atom.
\end{definition}

\begin{remark} \label{Rem:TwoLevel}
Let $L$ be a complete lattice. If $\mathcal{J}$ has at most two levels,
then clearly $\mathcal{J}$ does not contain a chain of three (or more) elements. 

If the lattice $L$ is spatial, then these conditions are equivalent.
Namely, assume that  $\mathcal{J}$ does not contain a chain of three elements.
Let $x < y$ be a maximal chain in $\mathcal{J}$. Suppose by contradiction that $x$ is not
an atom. Then there is $z \in L$ with $0 < z < x$. Since $L$ is spatial, there is  $j \in \mathcal{J}$ such 
that $j \leq z$. Now $j < x < y$ is a chain in $\mathcal{J}$ of three elements, 
a contradiction.
\end{remark}

Let $L$ be a complete lattice. It is well known that if $j$ is a completely join-irreducible element, then $j$ covers 
exactly one element, the \emph{lower cover} of $j$. We denote this element by $j_\prec$. Obviously,
\[
j_{\prec} = \bigvee \{x \in L \mid x<j \}. 
\]
It is clear that $j \in \mathcal{J}$ is an atom if and only if $j_{\prec} = 0$.

\begin{lemma}\label{Lem:Spatial2Level}
Let $L$ be a spatial lattice such that $\mathcal{J}$ has at most two levels. 
\begin{enumerate}[\rm (i)]
 \item If $j \in \mathcal{J} \setminus \mathcal{A}$, then $j_\prec$ is a join of atoms.
 \item The lattice $L$ is atomic.
\end{enumerate}
\end{lemma}

\begin{proof} (i) Let $j \in \mathcal{J} \setminus \mathcal{A}$. Since $L$ is spatial, 
$j_\prec = \bigvee \{ x \in \mathcal{J} \mid x < j\}$. Because $\mathcal{J}$ has at most two levels,
each $x \in \mathcal{J}$ such that $x < j$ is an atom. Therefore, $j_\prec$ is a join of atoms.

(ii) Since $L$ is spatial, we need to show only that there is an atom below each $j \in \mathcal{J}$.
If $j$ is an atom, we have nothing to prove. Now let $j \in \mathcal{J} \setminus \mathcal{A}$. 
Since $j_\prec \neq 0$ is a join of atoms by (i), there must be an atom below $j$,
\end{proof}

\begin{proposition} \label{Prop:Regular}
Let $(L,\vee,\wedge,{\sim}, {^*}, 0, 1)$ be a pseudocomplemented De Morgan algebra defined on an algebraic lattice. 
The following are equivalent:
\begin{enumerate}[\rm (i)]
 \item $L$ is regular;
 \item $\mathcal{J}$ has at most two levels.
\end{enumerate}
\end{proposition}

\begin{proof} (i)$\Rightarrow$(ii): Any pseudocomplemented De~Morgan algebra defines a
distributive double $p$-algebra in which the dual pseudocomplement is defined as in \eqref{Eq:Pseudocomplements}.
By Proposition~\ref{Prop:Varlet} there is no 3-element chain in $\mathcal{F}_p$. 
If $j,k \in \mathcal{J}$ with $j < k$, then $[j)$ and $[k)$ are prime filters such that $[k) \subset [j)$. 
By using Fact~\ref{Fact:TwoLevel}, we get that $[j)$ is a maximal prime filter. 
Then $j$ is an atom of $L$ and $\mathcal{J}$ has at most two levels.

(ii)$\Rightarrow$(i): We show that $x^* = y^*$ and $x^+ = y^+$ imply ${\sim} x = {\sim} y$. Notice that ${\sim} x = {\sim} y$ is equivalent to $x = y$.
Because $L$ is spatial by Remark~\ref{Rem:RingOfSets}, we have ${\sim} x = \bigvee \{ j \in \mathcal{J} \mid j \leq {\sim} x\}$ and
${\sim} y = \bigvee \{ j \in \mathcal{J} \mid j \leq {\sim} y\}$. It suffices to show that for any $j \in \mathcal{J}$,
$j \leq {\sim} x \iff j \leq {\sim} y$. Suppose for this that $j \leq {\sim} x$.

If $j \in \mathcal{A}$, then we must have $j \leq {\sim} y$. Otherwise $j \wedge {\sim} y = 0$, which further implies 
$j \leq ({\sim}y)^* = {\sim} y^+ =  {\sim} x^+ = ({\sim}x)^*$ by (2.5). Hence, we would get $j = j \wedge {\sim} x = 0$,
a contradiction.

If $j \in \mathcal{J} \setminus \mathcal{A}$, then there is $a \in \mathcal{A}$ such that $a < j$. This is because $L$ is
atomic by Lemma~\ref{Lem:Spatial2Level}. Because $L$  is by Remark~\ref{Rem:RingOfSets} completely distributive, we may define
the map $g \colon \mathcal{J} \to \mathcal{J}$ as in  \eqref{Eq:Gee}.
We have $g(j) < g(a)$. Since $\mathcal{J}$ has at most two levels, $g(j)$ is an atom. 

Now $j \leq {\sim} x$ yields $g(j) \not \leq x$ by \eqref{Eq:Connection}.
This implies $g(j) \wedge x = 0$ since $g(j)$ is an atom. Thus, $g(j) \leq x^* = y^*$. This gives
$g(j) \wedge y = 0$ and $g(j) \not \leq y$. By using \eqref{Eq:Connection} again, we obtain $j \leq {\sim} y$.

We have now proved that $j \leq {\sim} x$ implies $j \leq {\sim} y$. The converse can be proved symmetrically. Hence,
for any $j \in \mathcal{J}$, $j \leq {\sim} x \iff j \leq {\sim} y$ as required. 
\end{proof}

Let $(L,\vee,\wedge,{\sim},0,1)$ be a Kleene algebra defined on an algebraic lattice. 
By Proposition~\ref{Prop:Regular}, the pseudocomplemented Kleene algebra $\mathbb{L} = (L,\vee,\wedge,{\sim},{^*},0,1)$ is regular  if and only if 
$\mathcal{J}$ has at most two levels. Therefore, if $\mathbb{L}$ is regular, $\mathcal{J}$ can be trivially divided into two disjoint parts: 
the atoms $\mathcal{A}$ and the non-atoms $\mathcal{J} \setminus \mathcal{A}$.
On the other hand, by (J1)--(J3), the map $g \colon \mathcal{J} \to \mathcal{J}$ is an order-isomorphism between $(\mathcal{J},\leq)$ and 
$(\mathcal{J},\geq)$  such that each element $x$ in $\mathcal{J}$ is comparable with $g(x) \in \mathcal{J}$. This means that $\mathcal{J}$ can be divided into three
disjoint parts in terms of $g$: $\{ x \in \mathcal{J} \mid x < g(x)\}$, $\{ x \in \mathcal{J} \mid x = g(x)\}$, and 
$\{ x \in  \mathcal{J} \mid x > g(x)\}$.
We can write the following lemma connecting these two different ways to partition $\mathcal{J}$.

\begin{lemma} \label{Lem:PartitionOfAtoms}
Let $(L,\vee,\wedge,{\sim},{^*},0,1)$  be a regular pseudocomplemented Kleene algebra defined on an algebraic lattice. Then:
\begin{enumerate}[\rm (a)]
 \item $\mathcal{A} = \{x \in \mathcal{J} \mid x \leq g(x)\}$ \ and \ $\mathcal{J} \setminus \mathcal{A} = \{x \in \mathcal{J} \mid x > g(x)\}$.
 \item If $g(x) = x$, then $x$ is incomparable with other elements of $\mathcal{J}$. 
\end{enumerate}
\end{lemma}

\begin{proof}
(a) Let $x \in \mathcal{J}$. First, suppose $x \nleq g(x)$. Because $x$ and $g(x)$ are comparable, $x > g(x)$. The fact $g(x) \in \mathcal{J}$ means $g(x) \neq 0$. 
Then $0 < g(x) < x$ and $x \notin \mathcal{A}$. On the other hand, if $x < g(x)$, then because $\mathcal{J}$
has at most two levels, we have $x \in \mathcal{A}$. Finally, if $x = g(x)$ and $x \notin \mathcal{A}$, there
exists $y$ such that $0 < y < x$. Because $L$ is spatial, there exists $j \in \mathcal{J}$ such that $j \leq y < x$.
Therefore, $x = g(x) < g(j)$ and this yields that $x$ is an atom, a contradiction. 

(b) Suppose $g(x) = x$. By (a), $x$ is an atom, so there cannot be $y < x$ in $\mathcal{J}$. If $x \leq y$, then
$g(y) \leq g(x) = x$ gives $g(y) = x$, because $x$ is an atom. Hence $x = g(x) = g(g(y)) = y$. So, $x$ is comparable only with itself.
\end{proof}

Let $(L,\vee,\wedge,{\sim},{^*},0,1)$  be a regular pseudocomplemented Kleene algebra defined on an algebraic lattice.
Because $\mathcal{J}$ has at most two levels, $\mathcal{A}$ is the ``lower level'' and $\mathcal{J} \setminus \mathcal{A}$ is the ``upper level''.
For each $x$ in $\mathcal{J} \setminus \mathcal{A}$, the element $g(x)$ is an atom and $g(x) < x$. 
Obviously, $\mathcal{A}$ is an antichain, that is, any two elements in $\mathcal{A}$ are incomparable. This implies that also
$\mathcal{J} \setminus \mathcal{A}$ is an antichain, because if $x$ and $y$ are elements of $\mathcal{J} \setminus \mathcal{A}$ such that
$x < y$, then $g(x)$ and $g(y)$ are atoms and $g(x) > g(y)$ which is not possible. 

We define a relation $\simeq$ on $\mathcal{A}$ by:
\[ x \simeq y \iff x \leq g(y) . \]
Because each atom $x$ is such that $x \leq g(x)$ and $x \leq g(y)$ implies $y = g(g(y)) \leq g(x)$, the relation $\simeq$ is
a tolerance.  For any $x \in \mathcal{A}$, we denote
\begin{equation}\label{Eq:Span}
\langle x \rangle = \{ x \vee y \mid y \simeq x \} \cup \{ g(x)\}. 
\end{equation}

\begin{lemma}\label{Lem:Basis}
Let $x,y \in \mathcal{A}$.
\begin{enumerate}[\rm (a)]
 \item $y \in \langle x \rangle \iff g(y) \in \langle x \rangle \iff x = y$. 
 \item $\langle x \rangle = \{ x \} \iff g(x) = x$. 
 \item $\langle x \rangle \cap \langle y \rangle \neq \emptyset \iff x \simeq y$. 
\end{enumerate}
\end{lemma}

\begin{proof}
(a) The equivalences follow directly from the definition of $\langle x \rangle$.

(b) Because $g(x) \in \langle x \rangle$ by definition, $\langle x \rangle = \{x\}$ implies $g(x) = x$. On the other hand, 
if $g(x) = x$, then $x \simeq y$ implies $y \leq g(x) = x$. Because $x$ and $y$ are atoms, $x = y$. Thus $\langle x \rangle = \{x\}$.  

(c) If $x = y$, the claim is clear. Let $x \neq y$ and $x \simeq y$. Then $x \vee y \in \langle x \rangle \cap \langle y \rangle$. Conversely, assume 
$z \in \langle x \rangle \cap \langle y \rangle$. It is clear that $z$ is not an atom. Obviously, $z$ cannot be
of the form $g(a)$ for any atom $a$ neither, because $g(a)$ can belong only to $\langle a \rangle$. Thus $z \notin \mathcal{J}$.
Now $z \in \langle x \rangle$ implies $z = x \vee a$ for some some $a \simeq x$ and $z \in \langle y \rangle$ gives $z = y \vee b$ 
for some $b \simeq y$. Then $x = (y \vee b) \wedge x = (x \wedge y) \vee (x \wedge b)$. Because $x \neq y$ are atoms, we have
that $x \wedge y = 0$. Thus $x = x \wedge b$ which gives $x \leq b$. Because also $b$ is an atom, we have $x = b$ and
$x \simeq y$.
\end{proof}

Let us define $U = \bigcup \{ \langle x \rangle \mid x \in \mathcal{A}\}$. It is clear that the family 
$\mathcal{H} = \{ \langle x \rangle \mid x \in \mathcal{A} \}$
is an irredundant covering of $U$, because $x$ and $g(x)$ belong only to $\langle x \rangle$ for any $x \in \mathcal{A}$. 
We denote by $R$ the tolerance induced by $\mathcal{H}$. We have that for each $x \in \mathcal{A}$, the set $\langle x \rangle$
is a block of $R$. Because $R$ is induced by $\mathcal{H}$, we have 
\begin{equation} \label{Eq:R-sets}
  R(x) = \bigcup \{ \langle a \rangle \mid x \in   \langle a \rangle \} 
\end{equation}
for all $x \in U$. Since $\mathcal{A} \subseteq \mathcal{J} \subseteq U$, there are three kinds of elements in $U$. 
The following corollary is obvious by equations \eqref{Eq:Span} and \eqref{Eq:R-sets}.

\begin{corollary}\label{Cor:Neighbourhodds} Let $x,y,z$ be elements of $U = \bigcup \{ \langle a \rangle \mid a \in \mathcal{A}\}$.
\begin{enumerate}[\rm (a)]
 \item If $x \in \mathcal{A}$, then $R(x) = \langle x \rangle$.
 \item If $y \in \mathcal{J} \setminus \mathcal{A}$, then $y = g(a)$ for some $a \in \mathcal{A}$ and $R(y) = R(a) = \langle a \rangle$.
 \item If $z \in U \setminus \mathcal{J}$, then $R(z) = \langle a \rangle \cup \langle b \rangle$ for some distinct $a,b \in \mathcal{A}$ such
 that $z = a \vee b$.
\end{enumerate}
\end{corollary}

\noindent%
By applying the conditions of Corollary~\ref{Cor:Neighbourhodds} in Lemma~\ref{Lem:AppByCovering}, we can write for any $x \in \mathcal{J}$, 
\begin{equation}\label{Eq:UpperLowerGee}
R(x)^\DOWN = \{x, g(x)\} \text{ \ and \ } R(x)^\UP = \bigcup\{ \langle y \rangle \mid R(x) \cap  \langle y \rangle  \neq \emptyset \}.
\end{equation}
Lemma~\ref{Lem:Basis}(c) gives that for every $x \in \mathcal{A}$,
\begin{equation}\label{Eq:UpperGee}
R(x)^\UP = R(g(x))^\UP = \bigcup\{ \langle y \rangle \mid \langle x \rangle \cap \langle y \rangle \neq \emptyset \}
            = \bigcup\{ \langle y \rangle \mid x \simeq y \}.
\end{equation}

\section{Representation theorem} \label{Sec:Representation}

Let $\mathbb{L} = (L,\vee,\wedge,{\sim},{^*},0,1)$  be a regular pseudocomplemented Kleene algebra such that its underlying lattice is algebraic.
As in Section~\ref{Sec:Regularity}, we denote $\mathcal{H} = \{ \langle x \rangle \mid x \in \mathcal{A}\}$ and $U = \bigcup  \mathcal{H}$.
The tolerance $R$ is induced by the irredundant covering $\mathcal{H}$ of $U$ and the corresponding rough set lattice is denoted by $\mathit{RS}$. 
Let us agree that we denote $\mathcal{J}(L)$ simply by $\mathcal{J}$ and $\mathcal{J}(\mathit{RS})$ denotes the completely join-irreducible 
elements of $\mathit{RS}$.

For any $x \in \mathcal{J}$, we define
\[ \varphi(x) = \left \{ 
\begin{array}{ll}
 (\emptyset, R(x) ) & \text{if $x < g(x)$,} \\[1mm]
 (R(x)^\DOWN, R(x)^\UP) & \text{otherwise.}
\end{array} \right .
\]
If $x \in \mathcal{J}$, then $R(x)$ is a block. Indeed, if $x \in \mathcal{A}$, then $R(x) = \langle x \rangle$ is a block, 
and if $x \in \mathcal{J} \setminus \mathcal{A}$, then $g(x)$ is an atom and $R(x) = R(g(x)) = \langle g(x) \rangle$ is a block. 
Thus $(R(x)^\DOWN, R(x)^\UP) \in \mathcal{J}(RS)$ for every $x \in \mathcal{J}$ by Proposition~\ref{Prop:JoinIrreducibles}. Furthermore, if $x < g(x)$, then
$g(x) \in \langle x \rangle = R(x)$ and $|R(x)| \geq 2$. Therefore, $(\emptyset, R(x) ) \in \mathcal{J}(RS)$.
This means that the map $\varphi \colon \mathcal{J} \to  \mathcal{J}(RS)$ is well defined. Note also that if $x = g(x)$, 
then $x \in \mathcal{A}$ and $R(x) = \langle x \rangle = \{x\}$. This gives that $R(x)^\DOWN = R(x)^\UP = \{x\}$ and
$\varphi(x) = (\{x\},\{x\})$.

\begin{lemma}\label{Lem:order_embedding} 
The map $\varphi \colon \mathcal{J} \to \mathcal{J}(RS)$ is an order-isomorphism.
\end{lemma}

\begin{proof}
First, we show that $x \leq y$ implies $\varphi(x) \leq \varphi(y)$. If $x = y$, then trivially $\varphi(x) = \varphi(y)$.
If $x < y$, then $x \in \mathcal{A}$, $y \in \mathcal{J} \setminus \mathcal{A}$, and
$g(y) \in \mathcal{A}$. Therefore, $x < g(x)$ and $g(y) < y$ which imply
$\varphi(x) = ( \emptyset, R(x) )$ and $\varphi(y) =  (R(y)^\DOWN, R(y)^\UP)$.
By \eqref{Eq:UpperGee}, $R(y)^\UP = R(g(y))^\UP = \bigcup\{ \langle z \rangle \mid z \simeq g(y) \}$.
Since $x \leq y = g(g(y))$, we have $x \simeq g(y)$ and 
$R(x) = \langle x \rangle \subseteq  \bigcup \{ \langle z \rangle \mid z \simeq g(y) \} = R(y)^\UP$, implying $\varphi(x) \leq \varphi(y)$.
 
\medskip%

Secondly, we show that $\varphi(x) \leq \varphi(y)$ implies $x \leq y$. We begin by noting that if $\varphi(x) \leq \varphi(y)$,
then $x = g(x)$ and $y = g(y)$ are equivalent, and they imply $x = y$. To see this, suppose that $x = g(x)$. 
Then $\varphi(x) = (\{x\},\{x\})$. Now $\varphi(y) = (\emptyset, \langle y \rangle)$ is not possible, because we have
assumed that $\varphi(x) \leq \varphi(y)$. Therefore, we must have that $\varphi(y) = ( R(y)^\DOWN,R(y)^\UP)$.
This gives $x \in R(y)^\DOWN = \{y,g(y)\}$ by \eqref{Eq:UpperLowerGee}, so that $x = y$ or $x = g(y)$. Also the second equality gives
$x = g(x) = g(g(y)) = y$. Analogously, $g(y) = y$ means $\varphi(y) = (\{y\},\{y\})$. If $\varphi(x) = (\emptyset, \langle x \rangle)$,
then $x \in \langle x \rangle \subseteq \{y\}$ implies $x = y$ and $g(x) = g(y) = y = x$. If $\varphi(x) = (R(x)^\DOWN, R(x)^\UP)$, then
$R(x)^\DOWN = \{x,g(x)\} \subseteq \{y\}$ gives $x = g(x) = y$.

Therefore, we may assume $x \neq g(x)$ and $y \neq g(y)$. We divide the rest of the proof into four cases:
\begin{enumerate}[(i)]
 \item $x < g(x)$ and $y < g(y)$;
 \item $x < g(x)$ and $y > g(y)$;
 \item $x > g(x)$ and $y < g(y)$;
 \item $x > g(x)$ and $y > g(y)$.
\end{enumerate}

(i) Let $x < g(x)$ and $y < g(y)$. Then $x,y \in \mathcal{A}$ which yields $\varphi(x) = (\emptyset, R(x) ) = (\emptyset, \langle x \rangle)$ 
and $\varphi(y) = (\emptyset, R(y) ) = (\emptyset, \langle y \rangle)$.
By $\varphi(x) \leq \varphi(y)$ we get $x \in \langle x \rangle \subseteq  \langle y \rangle$, which is possible only if $x = y$ by Lemma~\ref{Lem:Basis}.

(ii) Suppose that $x < g(x)$ and $y > g(y)$. Hence $x$ and $g(y)$ are atoms. We have that
$\varphi(x) = (\emptyset, R(x)) = (\emptyset,\langle x \rangle)$ and $\varphi(y) = \{ R(y)^\DOWN, R(y)^\UP\}$.
Now $R(y)^\UP = R(g(y))^\UP = \bigcup \{ \langle z \rangle \mid z \simeq g(y) \}$.
Because $x \in \langle x \rangle \subseteq  \bigcup \{ \langle z \rangle \mid z \simeq g(y) \}$, we get
$x \in \langle z \rangle$ for some $z \simeq g(y)$. Then $z \leq g(g(y)) = y$. Since $x$ and $z$ are atoms, we obtain
$x = z$, and therefore $x \leq y$.

(iii) If $x > g(x)$ and $y < g(y)$, then $y \in \mathcal{A}$ and $x \in \mathcal{J} \setminus \mathcal{A}$.
Therefore, $\varphi(x) = (R(x)^\DOWN,R(x)^\UP)$ and $\varphi(y) = (\emptyset, R(y) )$.
Now $R(x)^\DOWN = \{x, g(x)\} \neq \emptyset$, contradicting $\varphi(x) \leq \varphi(y)$. 
Hence this case is not possible.

(iv) Assume $x > g(x)$ and $y >g(y)$. Then $g(x)$ and $g(y)$ are atoms and $x,y \in \mathcal{J} \setminus \mathcal{A}$.
We have that $\varphi(x) = ( \{x,g(x)\}, R(x)^\UP)$ and
$\varphi(x) = ( \{y,g(y)\},  R(y)^\DOWN\}$. By $\varphi(x) \leq \varphi(y)$ we have $\{x,g(x)\} \subseteq \{y,g(y)\}$.
If $x = y$, then we are done, and if $x = g(y)$, then $x < y$ because $g(y) < y$. 

\medskip%
Finally, we show that the map $\varphi$ is onto $\mathcal{J}(\textit{RS})$. Because $R$ is induced by the irredundant covering
$\{ \langle x \rangle \mid x \in \mathcal{A} \}$, there are two kinds of elements in $\mathcal{J}(\textit{RS})$: 
for each $x \in \mathcal{A}$ there is the rough set $(\langle x \rangle^\DOWN, \langle x \rangle^\UP)$, and for each
$x \in \mathcal{A}$ such that $|R(x)| = |\langle x \rangle| \geq 2$ there exists $(\emptyset, \langle x \rangle)$ in $\mathcal{J}(\textit{RS})$. So,
if $j = (\emptyset, \langle x \rangle)$ then $\varphi(x) = j$. Suppose $j = (\langle x \rangle^\DOWN, \langle x \rangle^\UP)$ for some $x \in \mathcal{A}$.
Because $x \in \mathcal{A}$, then $x = g(x)$ or $x < g(x)$. If $x = g(x)$, then $\langle x \rangle = \{x\}$, $j = (\{x\},\{x\})$,
and $\varphi(x) = j$. If $x < g(x)$, then $\varphi(g(x)) = (R(g(x))^\DOWN, R(g(x))^\UP) = (R(x)^\DOWN, R(x)^\UP) = j$.
\end{proof}

\begin{lemma}\label{Lem:CompatibleGee}
For all $x \in \mathcal{J}$, $\varphi(g(x)) = g(\varphi(x))$.
\end{lemma}

\begin{proof}
Because $\mathbb{L}$ is a completely distributive Kleene algebra,  there are by (J1)--(J3) three kinds of elements $x$ in $\mathcal{J}$
with respect to the map $g$: (i) $x < g(x)$, (ii) $x = g(x)$, and (iii) $x > g(x)$. Based on this, we divide the proof into three cases.

(i) If $x < g(x)$, then $\varphi(x) = (\emptyset, R(x))$ and $|R(x)| = |R(g(x))| \geq 2$. Therefore,
\[ \varphi(g(x)) = (R(g(x))^\DOWN, R(g(x))^\UP) = (R(x)^\DOWN, R(x)^\UP) = g((\emptyset, R(x))) = g(\varphi(x)). \]

(ii) If $x = g(x)$, then $R(x) = \{x\}$ and
\[ \varphi(g(x)) = \varphi(x) = ( \{x\},\{x\} ) = g(\varphi(x)) . \]

(iii) If $x > g(x)$, then $g(x) < g(g(x))$, and as in (i), $|R(x)| = |R(g(x))| \geq 2$. Thus
\[ 
\varphi(g(x)) = (\emptyset, R(g(x))) = (\emptyset, R(x)) = g(( R(x)^\DOWN, R(x)^\UP)) = g( \varphi(x)). \qedhere
\]
\end{proof}

We proved in \cite[Corollary 2.4]{JarRad11} that if $\mathbb{L} = (L,\vee,\wedge, {\sim},0,1)$ and 
$\mathbb{K} = (K,\vee,\wedge, {\sim},0,1)$ are two De~Morgan algebras defined on algebraic lattices
and $\varphi \colon \mathcal{J}(L) \to \mathcal{J}(K)$ is an order-isomorphism such that
$\varphi(g(j)) = g(\varphi(j))$ for all $j \in \mathcal{J}(L)$, then the algebras $\mathbb{L}$ and
$\mathbb{K}$ are isomorphic. By Lemmas~\ref{Lem:order_embedding} and \ref{Lem:CompatibleGee}, 
we can establish the following representation result.

\begin{theorem} \label{The:Representation}
Let $\mathbb{L} = (L,\vee,\wedge,{\sim}, {^*}, 0,1)$ be a regular pseudocomplemented Kleene algebra defined on an algebraic lattice. There
exists a set $U$ and a tolerance $R$ on $U$ such that $\mathbb{L} \cong \mathbb{RS}$.
\end{theorem}

\begin{example}
Let us consider the Kleene algebra $\mathbb{L}$ depicted in Figure~\ref{Fig:Figure1}.
\begin{figure}[ht]
 \centering
\includegraphics[width=96mm]{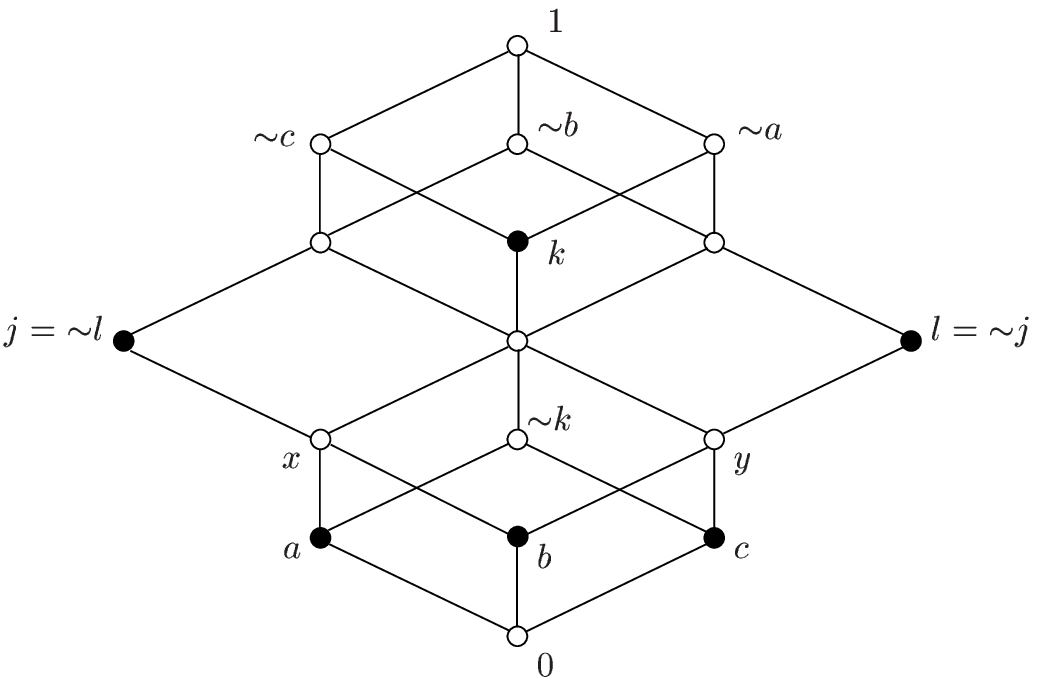}
\caption{\label{Fig:Figure1}}
\end{figure}

The (completely) join-irreducible elements are marked with filled circles. Now $\mathcal{A} = \{a,b,c\}$ and
$\mathcal{J} \setminus \mathcal{A} = \{j,k,l\}$. Note that the elements ${\sim} x$ are denoted for each $x \in \mathcal{J}$.
Because the lattice $L$ is finite, it is algebraic. It is easy to observe that
$\mathcal{J}$ has at most two levels. Therefore, the pseudocomplemented Kleene algebra $\mathbb{L}$ is regular. By Theorem~\ref{The:Representation},
there exists a set $U$ and a tolerance $R$ on $U$ such that the Kleene algebra $\mathbb{RS}$ determined by $R$ is isomorphic to $\mathbb{L}$. 
Next, we will illustrate this construction.

Now $g(a) = j$, $g(b) = k$, and $g(c) = l$. This means that the tolerance $\simeq$ on $\mathcal{A}$ is such that 
$a \simeq b$ and $b \simeq c$. For simplicity, we denote $a \vee b$ by $x$ and $b \vee c$ by $y$.  The sets
\[ \langle a \rangle = \{a,j,x\}, \langle b \rangle = \{b,k,x,y\}, \langle c \rangle = \{c,l,y\}  \]
form an irredundant covering of $U = \{a,b,c,j,k,l,x,y\}$ inducing $R$. We have
\begin{align*}
 R(a) = R(j) = \langle a \rangle;                  &&  R(b) = R(k) = \langle b \rangle; && R(c) = R(j) = \langle c \rangle ; \\
 R(x) =  \langle a \rangle \cup  \langle b \rangle; &&  R(y) =  \langle b \rangle \cup  \langle c \rangle . &&
\end{align*}

The rough set system $\mathit{RS}$ induced by the tolerance $R$ is depicted in Figure~\ref{Fig:Figure3}. The original Kleene algebra
$\mathbb{L}$ is isomorphic to the Kleene algebra $\mathbb{RS}$.
\begin{figure}[ht]
 \centering
\includegraphics[width=100mm]{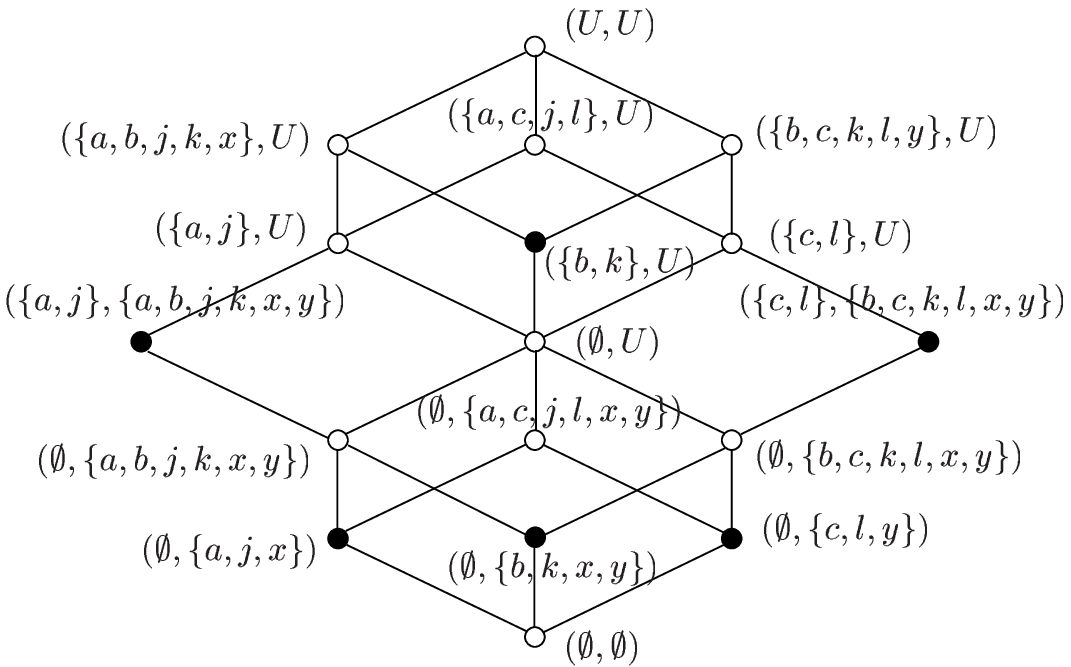}
\caption{\label{Fig:Figure3}}
\end{figure}
\end{example}

\section{Some concluding remarks} \label{Sec:Remarks}

Rough set lattices determined by quasiorders and by tolerances induced by irredundant coverings form Kleene algebras such that their underlying lattices are algebraic.
Their set of completely join-irreducible elements $\mathcal{J}$ are such that each $x \in \mathcal{J}$ is comparable with $g(x) \in \mathcal{J}$.
In this work we have shown that rough set algebras determined by tolerances induced by irredundant coverings are such that
$\mathcal{J}$ has at most two levels, and we proved that in case of pseudocomplemented De~Morgan algebras defined on
algebraic lattices these are exactly the (congruence-)regular ones. In case of an equivalence $E$, the set $\mathcal{J}$ of $\mathit{RS}$
is such that each $x \in \mathcal{J}$ is comparable \textit{only} with $g(x)$. This means that $\mathit{RS}$ is isomorphic to
$\mathbf{2}^I \times \mathbf{3}^K$, where $I$ is the set of singleton $E$-classes and $K$ is the set of $E$-classes having at least two elements.
The regular distributive double pseudocomplemented lattice $\mathit{RS}$ defined by an equivalence is in fact a regular double Stone algebra, 
and each regular double Stone algebra isomorphic to a product of chains $\mathbf{2}$'s and $\mathbf{3}$'s defines an equivalence $E$
such that the rough set algebra $\mathbb{RS}$ is isomorphic to the original regular double Stone algebra.

Obviously, we may divide the class of Kleene algebras defined on algebraic lattices into two classes: the ones in which $\mathcal{J}$ has at most
two levels, and those whose $\mathcal{J}$ has at least three levels. As we have shown in this work, if $\mathcal{J}$ has at most two levels,
then these algebras can be represented by tolerances which are induced by an irredundant covering. On the other hand, consider a Kleene algebra $L$ defined 
on an algebraic lattice such that there are at least three levels in $\mathcal{J}$. Now we may apply the results of \cite{JarRad11}, where we proved 
that the ones corresponding to rough sets determined by quasiorders are exactly those which satisfy the \emph{interpolation property}: 
if $x,y \leq g(x),g(y)$ for some $x,y \in \mathcal{J}$, then there exists $z \in \mathcal{J}$ 
such that  $x,y \leq z \leq g(x),g(y)$. Note that rough set systems defined by equivalences satisfy trivially this interpolation property, 
because each $x \in \mathcal{J}$ is comparable only with $g(x)$. Therefore, the condition  $x,y \leq g(x),g(y)$ is never true for $x \neq y$. 
Note also that we showed in \cite[Example 4.4]{JarRad11} that the height of $\mathcal{J}$ can be arbitrarily high. 

In the future it would be interesting to study what other kind of rough set structures can be characterized at the class of Kleene algebras defined on algebraic
lattices, by defining conditions on the set $\mathcal{J}$ of completely join-irreducible elements.

\section*{Acknowledgement} 
We would like to thank the anonymous referee who has put significant time and effort to provide expert views on our original manuscript.
In particular, pointing out an error in a proof and remarks on terminology are gratefully acknowledged. 


\providecommand{\bysame}{\leavevmode\hbox to3em{\hrulefill}\thinspace}
\providecommand{\MR}{\relax\ifhmode\unskip\space\fi MR }
\providecommand{\MRhref}[2]{%
  \href{http://www.ams.org/mathscinet-getitem?mr=#1}{#2}
}
\providecommand{\href}[2]{#2}

\end{document}